\documentclass[a4paper,12pt,leqno]{article}

\usepackage{authblk}

\usepackage{fullpage}
\usepackage{bbm}
\usepackage{graphicx}
\usepackage{upref}
\usepackage{enumitem}
\usepackage{latexsym}
\usepackage[colorlinks]{hyperref}
\usepackage{color,graphics}
\usepackage{comment} 

\usepackage{amsmath,amsthm,amssymb,algorithm,algorithmic,epsfig,graphicx}

\usepackage{amsmath}
\usepackage{amsfonts}
\usepackage{tikz} 

\usepackage{amssymb}
\usepackage{varioref}
\usepackage[ansinew]{inputenc}

\newtheorem{theorem}{Theorem}[section]
\newtheorem{lemma}[theorem]{Lemma}

\newtheorem{defn}[theorem]{Definition}

\newcommand{\lxy}{{\mathcal{M}}_d}
\newcommand{\mcx}{\C^d}
\newcommand{\mcy}{\C^d}

\newcommand{\R}{\mathbb{R}}

\newcommand{\C}{\mathbb{C}}

\newcommand{\X}{\mathcal{X}}

\renewcommand{\part}[2]{\frac{\partial #1}{\partial #2}}

\AtBeginDocument{}

\newcommand{\tr}{\mathrm{Tr}}
\newcommand{\vecc}{\mathrm{vec}}
\newcommand{\inner}[2]{\langle #1, #2 \rangle}

\newcommand{\Span}{\mathrm{span}}

\newcommand{\ext}{\mathrm{ext}}

\newcommand{\la}{\langle}
\newcommand{\ra}{\rangle}

\def\be{\begin{equation}}
\def\ee{\end{equation}}
\def\bi{\begin{itemize}}
\def\ei{\end{itemize}}

\newcommand{\tpsi}{\phi}
\newcommand{\calS}{\mathcal{S}}

\newcommand{\gram}{\mathrm{Gram}}

\newcommand{\rank}{\mathrm{rank}}

\newcommand{\CPSD}{\mathcal{CS}_+}
\newcommand{\CP}{\mathcal{CP}}

\newcommand{\calh}{\mathcal{H}}
\newcommand{\calE}{\mathcal{E}}
\newcommand{\sfT}{{\sf T}}
\newcommand{\CPSDR}{\mathrm{cpsd}\textnormal{-rank}}
\newcommand{\CPR}{\mathrm{cp}\textnormal{-rank}}

\newcommand{\cornm}{\pi(\calE_{n+m})}

\begin{document}
\title{Correlation matrices, Clifford algebras, and completely positive semidefinite rank}
 
%
       \author{Anupam Prakash\thanks{Email: aprakash@ntu.edu.sg}  \;\;and Antonios Varvitsiotis\thanks{Email: avarvits@gmail.com}}
 \affil{Nanyang Technological University, Singapore and Centre for Quantum Technologies, Singapore}
   

\maketitle

\begin{abstract} 
A symmetric $n\times n$ matrix $X$ is  completely positive
semidefinite (cpsd) if there  exist $d\times d$  positive
semidefinite {matrices} $\{P_i\}_{i=1}^n$ (for some $d\in \mathbb{N}$) such
that  $X_{ij}= \tr(P_iP_j),$ for all {$i,j \in \{ 1, \ldots, n \}$}.
The $\CPSDR$ of a cpsd matrix is the smallest~${d\in\mathbb{N}}$ for which
such a representation is possible. It was shown independently in \cite{CPSD} and \cite{cpsd2}  that there exist completely positive semidefinite  matrices with sub-exponential cpsd-rank. Both proofs were obtained using   fundamental  results  from the quantum information literature   as a  black-box.   In this work we give a   self-contained and succinct  proof of the existence of  completely positive semidefinite  matrices with sub-exponential cpsd-rank. For this, we introduce   matrix valued Gram decompositions for correlation matrices and show  that for extremal correlations, the matrices in such a   factorization  generate  a Clifford algebra.
Lastly, we show that this  fact underlies and generalizes  Tsirelson's results  concerning  the structure
of  quantum representations  for  extremal quantum correlation matrices.

\end{abstract}

\section{Introduction}

A symmetric  $n\times n$ matrix $X$ is  {\em completely positive
semidefinite} (cpsd) if there  exist $d\times d$ Hermitian positive
semidefinite {matrices} $\{P_i\}_{i=1}^n$ (for some integer  $d\ge 1$) satisfying   
\be\label{cpsd}
X_{ij}= \tr(P_iP_j), \text{ for all } i,j \in [n].
\ee
The set of $n\times n$ cpsd matrices forms  a convex  cone denoted by $\CPSD^n$. 
The cpsd  cone was   introduced recently  to provide linear conic formulations for the quantum analogues of various classical   graph parameters~\cite{LP14,R14b}.  Subsuming these results, it was  shown  in \cite{SV} that the set of joint probability  distributions   that can be generated using quantum resources  can be expressed as the projection of an affine section  of the $\CPSD^n$ cone. 

The {\em completely positive semidefinite rank}  of  $X\in \CPSD^n$, denoted by  $\CPSDR(X)$,  is defined as the {least} $d\ge 1$ for which there exist $d\times d$ Hermitian positive semidefinite  matrices  $P_1,\ldots,P_n$  satisfying   $X_{ij}=\tr(P_iP_j), $\  for all $i,j=1,\ldots,n$. The study of the $\CPSDR$   also has strong physical motivation  as it captures   the  size of a quantum system that is necessary to generate 
 a quantum probability distribution~\cite{CPSD}.

Besides their physical motivation, the cpsd cone and the cpsd-rank are also  interesting from  the perspective of linear conic optimization. Firstly, the cpsd cone and the cpsd-rank are  non-commutative analogues  of the completely positive cone and the $\CPR$. A 
symmetric $n\times  n $ matrix $X$ is  {\em completely
positive} (cp)  if there exist  vectors $\{p_i\}_{i=1}^n\subseteq \R^d_+$ (for some $d\ge 1$) such that $X_{ij}=
p_i^\sfT p_j,$ for all $ i,j=1,\ldots,n$. The {\em cp-rank} of a cp matrix $X$ is the least integer $d\ge 1$ for which such a factorization is possible.   The cp cone  has been extensively studied as any  quadratic program with a mix of     binary and continuous variables can be formulated  as a linear  conic  program over the cp cone~\cite{Bur07}.

 Secondly,  cpsd factorizations correspond to symmetric psd-factorizations. Recall that a {\em $d$-dimensional  psd factorization} of  a  matrix $X\in \R^{n\times m}_+$   consists of  two families  of  $d\times d$ Hermitian positive
semidefinite {matrices} $\{P_i\}_{i=1}^n, \{Q_j\}_{j=1}^m$ 
satisfying   
$X_{ij}= \tr(P_iQ_j)$  { for all } $i=1,\ldots,n$ and $  j=1,\ldots,m.$ The {\em psd-rank} of an entrywise nonnegative matrix $X$ is the smallest dimension  of a psd factorization.  The psd-rank has been extensively studied as it characterizes the semidefinite extension complexity of polytopes  and as a result, captures the efficacy of semidefinite programming for   hard combinatorial problems~\cite{FGPRT}.

 Combining  Carath\'eodory's
 Theorem  (e.g. see  \cite[Theorem 1.34]{CP}) with  the atomic reformulation of the $\CPR$ we   have  that   $\CPR(X)\le \binom{n+1}{2}$ for  any   $X\in\CP^n$. The   best upper bound currently is     ${n^2\over 2}+O(n^{3/2})$
 \cite{BSU}, which   is asymptotically  tight with respect to the  Drew-Johnson-Loewy  lower bound of $\big\lfloor {n^2/ 4}\big\rfloor,$ for $n\ge 4$ \cite{DJL}. 
 Furthermore, the psd-rank of a  matrix  $X\in \R^{n\times m}_{+}$  is always at most $\min\{n,m\}$ as there is a factorization $X_{ij} = \tr(P_i Q_j)$ 
 for diagonal matrices $P_{i}=\text{diag}(e_{i})$ and $Q=\text{diag}(X^{j})$ where $e_{i}, i \in [n]$ are the standard basis vectors for $\R^{n}$ and 
 $X^{j}$ is the $j$-th column  of~$X$. 
 

In contrast to these related notions of matrix ranks, it was  shown independently in \cite{CPSD} and  \cite{cpsd2},  that there exist  cpsd matrices  whose  cpsd-rank is sub-exponential in terms of their dimension. Specifically, we have that:


\begin{theorem}[\cite{CPSD,cpsd2}]\label{thm:cpsdlowerbound}
For  any $n\in \mathbb{N}$ there  exists a matrix   $X_n\in\CPSD^{2n}$ such~that
\be\label{cdwferfer}
 \CPSDR(X_n)\ge  {{2}^{\lfloor r_{\max}(n) / 2 \rfloor}},
 \ee
 where   $r_{\max}(n)$  is   the {greatest} integer
satisfying $\binom{r+1}{2}\le~n$, {\em i.e.},
$$r_{\max}(n)=\left\lfloor {\sqrt{1+8n}-1 \over 2}\right\rfloor.$$

 \end{theorem}

 Both   proofs of Theorem~\ref{thm:cpsdlowerbound} \cite{CPSD,cpsd2} were obtained using   fundamental  results  from the quantum information literature   as a  black-box~\cite{TS87}. Our main goal in this 
article is to give a self-contained proof of Theorem  \ref{thm:cpsdlowerbound}
 that  bypasses the  quantum information results that were used in  the original proofs \cite{CPSD, cpsd2}, and as such, makes it   accessible to the broader mathematical  community.  Furthermore, the new proof presented in this article  highlights    certain  matrix factorizations of  the elliptope as the main underlying   mathematical tool and paves the way for further generalizations.   



 \paragraph{{\bf Main technical result.}}
The {\em correlation matrix}  of  the  random variables $X_1, \ldots, X_n$ is the $n\times n$ matrix whose $(i,j)$ entry is equal to the correlation between $X_i$ and $X_j$, {\em  i.e.},
$$\mathbb{E}[(X_{i}  - \mu_{i}) (X_{j} - \mu_{j})]/\sigma_{i} \sigma_{j},$$ where $\mu_{i}, \sigma_{i}$ denote  the mean and standard deviation of $X_{i}$. Correlation 
matrices capture the association between random variables and their  use is  ubiquitous in statistics. 

It is easy to verify that correlation matrices are positive semidefinite and have all diagonal entries equal to one. Conversely, any such matrix can be expressed as a correlation matrix for some family of random variables.  Thus, the set of $n\times n$ correlation matrices coincides with  the $n$-dimensional {\em elliptope}, denoted by $\calE_n$,   defined as the set of $n\times n$ symmetric positive semidefinite  matrices with  diagonal entries   equal to one, {\em i.e.},  $$\calE_n:=\{E\succeq 0: E_{ii}=1\ (i\in [n])\}.$$
The elliptope is a spectrahedral set  whose structure has  been  extensively  studied ({\em e.g.} see~\cite{DL} and references therein). Its significance  is illustrated by the fact  that it corresponds to  the feasible region of various semidefinite programs that are used to approximate  NP-hard combinatorial optimization problems ({\em e.g.} MAX-CUT \cite{GW}).


  In  this work we introduce and study matrix factorizations of a specific form for correlation matrices. 
  Informally, our main  result  is that for {\em extreme points}  of the set of  correlation matrices, the matrices in such a factorization  generate a Clifford~algebra.  
  
  Recall that  the {\em rank-$n$ Clifford algebra}, denoted by $\mathcal{C}_n$, is the universal $C^\ast$-algebra generated by Hermitian indeterminates $z_1,\ldots, z_n$ satisfying the following relations:
\be 
z_iz_j+z_jz_i=2\delta_{ij}I, \quad \text{ for all }   i, j\in [n].
\ee
Furthermore, it is well-known that depending on the parity of $n$,   the algebra $\mathcal{C}_n$ has either one or two irreducible representations, each of dimension  $2^{\lfloor{ n/ 2}\rfloor}$, {\em e.g.},  see \cite[Chapter 6]{GW}.

 Having introduced Clifford algebras, we  now formally state our main technical result.

\begin{theorem}\label{thm:bipartite}
Let   $E$ be an extreme point of $\calE_{k}$ with $\rank(E)=n$  and let $A$ be a full-rank principal  submatrix of $E$. 
Assume that     $E=\left(\begin{smallmatrix}A& C\\C^\sfT& B\end{smallmatrix}\right)$ and consider  $d\times d$   Hermitian matrices    $\{X_i\}_{i=1}^n, \{Y_j\}_{j=1}^{k-n},  K$~satisfying
\begin{itemize}[itemsep=0.2em]
 \item[$(i)$]  $E={\rm Gram}(KX_1,\ldots,KX_{n}, Y_1K,\ldots,Y_{k-n}K)$;
\item[$(ii)$]    $X_i^2=Y_j^2=I_d,\  \text{ for all }  i\in [n], j\in [k-n]$;
\item[$(iii)$]   $  \tr(K^2)=1$ and $  K$ is  positive definite.
\end{itemize}
Then, the algebra $\mathbb{C}[ X_1,\ldots,X_n]$   is isomorphic to the rank-$n$ Clifford algebra $\mathcal{C}_{n}.$ In particular,  the size of the matrices  $X_1,\ldots, X_n$ is      lower bounded by~$2^{\lfloor n/2\rfloor}$.
%
\end{theorem}


   The proof of Theorem~\ref{thm:bipartite} is given in Section~\ref{thm:bipartite}. Throughout this work, we  refer to any family of matrices $\{X_i\}_{i=1}^n,\{Y_j\}_{j=1}^{k-n}, K$ satisfying conditions $(i), (ii)$ and~$(iii)$ above  as a {\em matrix factorization of the correlation matrix $E$}.

A couple of  comments are in order concerning  Theorem~\ref{thm:bipartite}. First,  it is well-known that there exists a  rank-$n$ extreme point of $\calE_{k}$  if and only if   $k\ge \binom{n}{2}$  (e.g. see Section~\ref{sec:basicproperties}),  and this condition will  be satisfied whenever we apply Theorem~\ref{thm:bipartite}.  Second, although not immediately obvious, we show  in Lemma~\ref{lem:representations} that  any correlation matrix admits  such a matrix factorization (where we can  even always take  $K$ to  be  a multiple of the identity). Third, it is worth noting that Theorem~\ref{thm:bipartite} remains valid  when  $(i)$ is replaced~with:
\bi 
\item[$(i')$] $E={\rm Gram}(KX_1,\ldots,KX_{n}, KY_1,\ldots,KY_{k-n}).$
\ei
and also 
when $(i), (ii)$ and $(iii)$ are replaced with: 
\begin{itemize}[itemsep=0.2em]
 \item[$(i'')$]  $E={\rm Gram}(A_1,\ldots,A_{n}, B_1,\ldots,B_{k-n})$;
\item[$(ii'')$]    $A_i^2=B_j^2={1\over d} I_d,\  \text{ for all }  i\in [n], j\in [k-n]$.
\end{itemize}


%
 
Using Theorem \ref{thm:bipartite}, in Section \ref{sec:cpsdrank} we give a short  proof of Theorem \ref{thm:cpsdlowerbound}. 

 \paragraph{{\bf Related work  from   quantum information theory.}} 
The image  of the  elliptope $\calE_{n+m}$  under the  projection~operator
\be\label{eq:projection}
\pi: \mathcal{S}_{n+m} \rightarrow  \R^{n\times m},  \quad  \begin{pmatrix}A& C\\C^\sfT& B\end{pmatrix} \mapsto C,
\ee
is known as the set  of {\em $(n,m)$ bipartite correlation matrices} 
and is of central importance to quantum information theory. The link   between quantum information theory and bipartite correlation matrices  is best explained within  the following  framework, known in the physics literature  as a {\em Bell scenario} \cite{AMO}. Consider two parties, Alice and Bob, that share a bipartite quantum system, {\em e.g.,} a pair of spin-1/2 particles. 
According to the postulates of quantum mechanics, the  {\em state} of a bipartite quantum  system with {\em local dimension $d$} is  described by a Hermitian psd matrix $\rho$ acting on~$\C^d\otimes \C^d$   with $\tr(\rho)=1$. 

Independently and simultaneously, Alice and Bob   measure  their   part of the system and then  announce the outcomes  of their measurements.  By the postulates of quantum mechanics,  the  process of measuring a $d$-dimensional quantum system  is described   by  a  Hermitian matrix $H \in \C^{d \times d}$,  called  an   {\em observable}.
    By the spectral theorem, for all observables $H$  we have the decomposition $H=\sum_{i=1}^k\lambda_iP_i,$ where  $\{ \lambda_i\}_{i=1}^k$ ($k\le d)$ are the eigenvalues  of $H$  and $\{P_i\}_{i=1}^k$ {are} the projectors {onto} the corresponding eigenspaces. The  measurement  defined by  $H$ has  possible outcomes $\{\lambda_i\}_{i=1}^k$. Upon measuring a 
state $\rho,$ the outcome $\lambda_i$ is observed with probability $\tr(\rho P_i)$. As a consequence, $\tr(\rho H)$ is the \emph{expectation} of  the outcome upon  performing the measurement $H$ on state $\rho$.
 
Throughout this paper, we {\em only} consider observables $H$  whose spectrum lies in $[-1,1]$, {\em i.e.,} observables satisfying  $H^2\preceq I$. Furthermore,  we assume throughout that  Alice has $n$ measurement choices given by the observables $M_1,\ldots, M_n$ and Bob has $m$ measurement choices  denoted  by $N_1,\ldots,N_m$. Accordingly,   if  their shared system  is in state $\rho$ and  Alice and Bob perform   measurements $M_i$ and $N_j$ respectively,  the expectation of the {\em product} of their answers is given by $\tr(\rho M_i\otimes N_j)$. The 
 $n\times n$ matrix whose $(i,j)$ entry is given by $\tr(\rho M_i\otimes N_j)$  is called a {\em quantum  correlation matrix}, {\em cf.} Definition~\ref{rgrgtr}. 


 Tsirelson showed in \cite{TS87}  that, in the two-outcome case,  the set of quantum  correlation matrices {\em coincides} with the set of bipartite correlation matrices ({\em cf.} Theorem~\ref{thm:tsirelson1appendix}). This  characterization has found numerous  applications. As one example,  this  implies that one can optimize a linear function over the set of quantum  correlation matrices (up to arbitrary precision) in polynomial-time, using semidefinite programming.

 Additionally, Tsirelson   studied   the structural properties  of the quantum state and the observables  that are necessary  to generate an extremal  quantum  correlation matrix. Roughly speaking, he showed that   any quantum  correlation matrix can be generated using observables $\{M_i\}_i$ and $\{N_j\}_j$,  where all the anti-commutators $M_iM_j+M_jM_i$  are scalar multiples of the identity  (and analogously for the $N_j$'s).    Moreover, he showed that for {\em extremal} quantum  correlation matrices, such representations are essentially the only possible ones, {\em e.g.} see \cite[Theorem 3.1]{TS87} and \cite[Theorem 3.8]{TS93}. Using standard results concerning the representations of Clifford algebras, Tsirelson's work  implies that the local dimension  of a quantum system necessary to generate an extreme quantum  correlation matrix can be lower bounded in terms of its rank. 
 
 \begin{theorem}[\cite{TS87,TS93}]\label{thm:lowerboundlocaldimension}
Given an extreme bipartite correlation matrix  $C\in \ext(\cornm)$, the  local dimension of any   tensor product representation  is lower bounded by~$2^{\lfloor \rank(C)/2\rfloor}$. 
\end{theorem} 
 
%

%

 Using  Theorem~\ref{thm:bipartite}, in Section \ref{sec:quantumcorrelations}   we derive  Theorem \ref{thm:lowerboundlocaldimension}.
For this  we show  that any extreme  quantum  correlation matrix can be completed in a unique way to a correlation matrix and also, this completion is an extreme  correlation matrix.  Thus, lower bounds  on matrix factorizations of extreme correlation matrices (such as Theorem~\ref{thm:bipartite}) imply lower bounds on  factorizations of extremal  quantum  correlation matrices (such as Theorem~\ref{thm:lowerboundlocaldimension}).

 The proof of Theorem~\ref{thm:bipartite}  relies on  ideas from   Tsirelson's work. Nevertheless,  Theorem~\ref{thm:bipartite} 
 strictly generalizes Theorem~\ref{thm:lowerboundlocaldimension}  as   the projection of an extreme point is not necessarily extreme, {\em i.e.},  there exist matrices    $\left(\begin{smallmatrix}A& C\\C^\sfT& B\end{smallmatrix}\right)\in\ext(\calE_{n+m})$ for which  $C\not \in \ext (\cornm)$. 
 A concrete  example of such a matrix is given at  the end of Section \ref{sec:quantumcorrelations}.

%

%

 \paragraph{{\bf Further related work and open problems.}} 
 Theorem~\ref{thm:bipartite} provides  an alternative interpretation of Tsirelson's results,  by   highlighting  matrix factorizations of the elliptope as the   underlying mathematical object.
On the other hand, representations of generalized  Clifford algebras (associated with a real zero
polynomial) are related  to the existence of determinantal  representations of (powers of) hyperbolic polynomials, {\em e.g.} see \cite{netz1, netz2}. It is an interesting question whether   representations of  generalized  Clifford algebras are related in a similar manner  to  other spectrahedra, more general than the~elliptope. 

Let  $\mathcal{K} = \{ K_{d} \}_{d}$ be a family of cones that is  closed under direct sums and define $${\rm Gram}(n, \mathcal{K})=\{{\rm Gram}(u_1,\ldots,u_n): u_1,\ldots,u_n \in K_d, \text{ for some } d\in \mathbb{N}\},$$
 which  is  itself a convex cone.  For a matrix $X\in {\rm Gram}(n, \mathcal{K} )$ define {\em Gram-rank(X)} to be the least  $d\ge 1$ such that  $X$ can be realized as a  Gram matrix of vectors in $K_d$. As noted earlier, for a matrix $X$ in ${\rm Gram}(n, \{ \R^{d} \}_{d} ) = \mathcal{S}^{n}_{+}$ and ${\rm Gram}(n, \{ \R^{d}_{+} \}_{d} ) = \mathcal{CP}_{n}$, the Gram-rank of $X$ is upper bounded by a polynomial in the dimension $n$ while for ${\rm Gram}(n, \{ \mathcal{S}^{d}_{+} \}_{d} ) = \mathcal{CS}^{n}_{+}$, Theorem \ref{thm:cpsdlowerbound} establishes a sub-exponential lower bound. 

An interesting research direction is to prove upper and lower bounds for {\em Gram-rank(X)} in terms of the dimension for matrices $X \in  {\rm Gram}(n, \mathcal{K})$ for 
other families of cones  and also, to investigate the cases for which there is a super polynomial lower bound. The closely related notion of  generalized completely positive cones  has been studied in \cite{gowda1}.

Lastly, we mention that following the completion of this work, there have been exciting new results  concerning the closedness of the set of quantum behaviors that imply  the existence of completely positive semidefinite matrices (of fixed size) with arbitrarily large cpsd-rank \cite{Slof,DPP,musat}.

\subsection{Preliminaries}In this section we  introduce the most important  definitions, notation and background material  that we use throughout  this paper. 

\medskip

%
%
%
%
%

\noindent {\bf  Linear Algebra. }
  Throughout, we  use the shorthand notation $\{x_i\}_{i=1}^n:=\{x_1,\ldots,x_n\}$ and $[n]:=\{1,\ldots,n\}$. We denote by $\{e_i\}_{i=1}^n$ the standard  basis of~$\C^n$.  The canonical  inner product of two vectors  $x,y\in \R^n$ is denoted by $\la x,y\ra$. We write $\Span({\{x_i\}_{i=1}^n})$ for the  linear span of the vectors $\{x_i\}_{i=1}^n$.
  
We denote by $\lxy$  the set of $d\times d$ complex matrices and by $\calh_d$ (resp. $\mathcal{S}_d$) the set of $d\times d$ Hermitian (resp. symmetric) matrices. Given  a matrix $X\in \lxy$, its transpose is denoted by $X^\sfT$ and its conjugate transpose by $X^*$. Furthermore, we denote by $X\otimes~Y$   the Kronecker product  of $X$ and $Y$. Throughout,  we
equip $\calh_d$ with the Hilbert-Schmidt inner product $\la
X,Y\ra:=\tr(XY^*)$.  For a block matrix   $X=\left(\begin{smallmatrix}A& C\\C^\sfT& B\end{smallmatrix}\right)\in~\mathcal{S}_{n+m}$ we use   that 
\be\label{eq:usefulllll}
\rank(X)=\rank(A) \Longleftrightarrow \exists\  n\times m  \text{ matrix } \Lambda  \text{ such that }C=A\Lambda \text{ and } B=\Lambda^\sfT A\Lambda.
\ee

A matrix $X\in\calh_d$ is called {\em positive semidefinite} (psd) if
$\psi^*X\psi\ge 0,$ for all $\psi\in \C^d$. The set of $d\times d$
Hermitian psd (resp. symmetric psd) matrices forms a closed convex
cone  denoted by $\calh^d_+$ (resp.~$\calS^d_+)$. We sometimes also  write $X\succeq 0$ to indicate that $X$ is~psd. 

The  {\em Gram matrix} of a family of vectors
$\{x_i\}_{i=1}^n\subseteq \R^d$, denoted by $\gram(x_1,\ldots,x_n)$ or 
$\gram({\{x_i\}_{i=1}^n}),$ is the  symmetric  $n\times n$ matrix whose
{$(i,j)$} entry is given by $\la x_i,x_j\ra$, for all $i,j\in [n].$ It is easy to see that  an $n\times n$ matrix $X$ is positive semidefinite if and only if there exist vectors  $x_1,\ldots,x_n\in \mathbb{R}^k$ (for some $k\ge 1$) such that $X=\gram(x_1,\ldots,x_n)$.  For any Gram matrix we have that $\rank
\left(\gram( \{x_i\}_{i=1}^n)\right)=\dim (
\Span(\{x_i\}_{i=1}^n))$. Lastly, 
 if  $X=\gram(x_1,\ldots,x_n)$ we make use the following property:
 \be\label{eq:useful}
 \text{For any }   \lambda \in \mathbb{R}^n: X\lambda=0 \Longleftrightarrow \sum_{i=1}^nx_i\lambda_i=0.
 \ee

We use  a well-known   correspondence between  $\lxy$  and $\mcy\otimes \mcx$ given by the  map
${\rm vec}:\lxy \rightarrow  \mcy\otimes \mcx,$ 
which is given by 
$\vecc(e_ie_j^*)=e_i\otimes e_j, $
on basis vectors and is extended linearly. The $\vecc(\cdot)$ map 
is  an   isometry, {\em i.e.}, $\inner{X}{Y} = \inner{\vecc(Y)}{\vecc(X)}$ for all $X,Y\in \lxy$. 
We also need the following fact: 
\begin{equation} \label{vecprop}
 \vecc(W)^*(X \otimes Y) \vecc(Z)= \vecc(W)^*\vecc(XZY^\sfT)=
\la W,XZY^\sfT\ra, 
\end{equation} 
 
Any  vector $\psi\in \mcy\otimes \mcx$ can be {uniquely} expressed as ${\psi=\sum_{i=1}^d\lambda_i \, y_i\otimes x_i} $ for 
some integer $d\ge 1$, positive  scalars $\{ \lambda_i\}_{i=1}^d$, and orthonormal sets ${\{ y_i\}_{i=1}^d\subseteq \mcy}$ and $\{ x_i\}_{i=1}^d\subseteq~\mcx$. An expression of this form is known as a {\em Schmidt decomposition} for $\psi$ and is derived by the singular value  decomposition of $\vecc^{-1}(\psi)$.
Note that if  ${\psi=\sum_{i=1}^d\lambda_i \, y_i\otimes x_i} $ is a Schmidt decomposition for $\psi$, then {we have that} $\|\psi \|^2_2 =~\sum_{i=1}^d \lambda_i^2$. 

The {\em Pauli matrices} are given by
\[
{I_2:=
\begin{pmatrix}
1 & 0 \\
0 & 1
\end{pmatrix}}, \;
X:=
\begin{pmatrix}
0 & 1 \\
1 & 0
\end{pmatrix}, \;
Y:=\begin{pmatrix}
0 & -i \\
i  & 0
\end{pmatrix},
\; \text{ and } \;
Z:=
\begin{pmatrix}
1 & 0 \\
0 & -1
\end{pmatrix}. \]
Note that the {(non-identity)} Pauli matrices are Hermitian, their trace is equal to zero,  they have $\pm 1$ eigenvalues and they pairwise anti-commute. 

\medskip 

\noindent {\bf  Clifford algebras. } 
Throughout this section set   $d:=2^{\lfloor{ r/ 2}\rfloor}.$ It  is well-known that
\be\label{eq:Clifford}
 \mathcal{C}_r \cong \mathcal{M}_d, \  \text{ for even } r, \  \text{ and }\  \mathcal{C}_r \cong \mathcal{M}_d\oplus  \mathcal{M}_d,  \text{  for odd } r.
  \ee
For a proof of this fact and additional details the reader is referred to~\cite[Chapter 6]{GW}.
An explicit  representation  of $\mathcal{C}_r$ is obtained using the Brauer-Weyl matrices. Specifically, for  $r=2\ell$, the map  
 $\gamma_r: \mathcal{C}_r \rightarrow \calh_d$ given by 
 \be \label{eq:firsthalf}
 \gamma_r(z_i)= Z^{\otimes(i-1)}\otimes X\otimes I_2^{\otimes (\ell-i)} \in \calh_d,  \ (i\in [\ell]),
 \ee
 and
 \be\label{eq:secondhalf}
 \gamma_r(z_{i+\ell})= Z^{\otimes(i-1)}\otimes Y\otimes I_2^{\otimes (\ell-i)}\in \calh_d, \ (i\in [\ell]).
\ee
is a complex representation of $\mathcal{C}_r$. Furthermore, in the case where    $r=2\ell+1$ we define $\{\gamma_r(z_i)\}_{i=1}^{2\ell}$ as described in \eqref{eq:firsthalf} and \eqref{eq:secondhalf}  and  additionally set $\gamma_r(z_{2\ell+1})= Z^{\otimes \ell}.$  

Lastly, we collect some properties of the  map $\gamma_r$ which will be crucial for our results  in the next section.   Specifically, setting $d=2^{\lfloor{ r/ 2}\rfloor},$ for  all $x,y\in \R^n$ we  have the following: 
 \be\label{eq:gamma}
  \gamma_r(x)^2=\|x\|^2I_d\ \quad  \text{ and } \quad  d \la x,y\ra=\tr\left( \gamma_r(x)\gamma_r(y)\right).
 \ee 

\noindent {\bf  Convexity. }A set $C \subseteq \R^{n}$ is {\em convex} if for all $a, b \in C$  and $\lambda \in [0,1]$ we have that  $\lambda a + (1-\lambda) b \in C$. 
A subset $F\subseteq C $ is called a {\em face} of $C$ if $\lambda c_1+(1-\lambda)c_2\in F$ implies that $c_1,c_2\in F$, for all $c_1,c_2\in C$ and $\lambda\in [0,1]$. We say that $c$ is an {\em extreme point} of the convex set $C$ if the singleton   $\{c\}$ is a face of $C$. 
We denote by ${\rm ext} (C)$ the set of extreme points of the convex set $C$. By the
Krein-Milman theorem, any compact convex subset of $\R^n$ is equal to the convex hull of its extreme points, {\em e.g.} see \cite{bar}. 
\section{Correlation matrices, extreme points, and matrix factorizations}
\subsection{Extreme correlation  matrices and quadratic maps}\label{sec:basicproperties}
An  {\em operator valued quadratic map} is a function      
$Q: \R^r\rightarrow \calh_d$ (for some $d\ge 1$) such that $Q(ax)=a^2Q(x),$ for all $x\in \R^r$ and $a\in \R$.
The following    result from \cite{TS87} will be crucial.

\begin{lemma}\label{lem:qforms}
Consider vectors  $\{u_i\}_{i=1}^n\subseteq \R^r$ satisfying  $\Span(u_1,\ldots,u_n)=\R^r$  and  
\be\label{extremepoints}
\Span\left(u_1u_1^\sfT,\ldots,u_nu_n^\sfT\right)=\mathcal{S}_r.
\ee
Then, for  any operator valued quadratic map $Q: \R^r\rightarrow \calh_d$ we have that 
\be\label{eq:quadraticform} 
Q(u_i)=0, \ \text{ for all }  i\in [n]\  \Longrightarrow \ Q=0.
\ee

\end{lemma}
\begin{proof} First consider the case $d=1$, {\em i.e.}, we have a quadratic form $Q: \R^r\rightarrow \R$. Let $M_Q\in \calS_r$ be the symmetric matrix, corresponding to the bilinear form associated  to $Q$, with respect to the standard basis $\{e_i\}_{i=1}^r$ of $\R^r$. By assumption \eqref{eq:quadraticform} we have that $Q(u_i)=u_i^\sfT M_Q u_i=0$ for all $i\in [n]$ and thus, \eqref{extremepoints} implies that $M_Q=0$. For the case $d>1$, since $\calh_d$ is a $d^2$-dimensional vector space over the real numbers, we can equivalently  view $Q$ as a quadratic form $Q: \R^r\rightarrow \R^{d^2}$ where $x\mapsto  (Q_1(x),\ldots,Q_{d^2}(x))$.   As all the $Q_i$'s are real valued quadratic forms, the proof is concluded  from the base~case.
\end{proof}
Next, we recall  the following well-known characterization of the extreme points of $\calE_n$.
 


\begin{theorem}[\cite{LT94}]\label{thm:extremepoints1}

Consider  $E\in\calE_n$ with $r=\rank(E)$ and let $E=\gram(\{u_i\}_{i=1}^n)$ where  $u_{i} \in \R^r$.  Then  $E\in \ext (\calE_n)$ if and only if 
\be\label{extremepoints2}
\Span\left(u_1u_1^\sfT,\ldots,u_nu_n^\sfT\right)=\mathcal{S}_r.
\ee
Equivalently, we have that $E\in \ext(\calE_n)$ if and only if 
\be\label{eq:extremality} 
\rank(E\circ E)=\binom{\rank(E)+1}{2},
\ee 
where  $X\circ Y$ denotes the entrywise product of   $X$ and $ Y$.

\end{theorem}

Combining Lemma \ref{lem:qforms} with Theorem \ref{thm:extremepoints1} we have the following useful result. 

\begin{theorem}\label{lem:equalityquadratic}
Consider  $E\in\ext(\calE_n)$ with $r=\rank(E)$ and let $E=\gram(\{u_i\}_{i=1}^n)$ where $\Span(\{u_i\}_{i=1}^n)=\R^r$.
For any    two operator valued quadratic maps  $Q_1,Q_2 : \R^r\rightarrow \calh_d $ (for some $d\ge 1)$  satisfying  $Q_1(u_i)=Q_2(u_i),$ for all $i\in [n]$   we have that $Q_1=Q_2$.
\end{theorem}
\begin{proof}Consider the operator valued quadratic form  $Q:=Q_1-Q_2$.  By assumption we have that $Q(u_i)=0$ for all $i\in [n]$.  As  $ E\in\ext (\calE_n)$, by Theorem \ref{thm:extremepoints1} we have that \eqref{extremepoints2} holds. 
Lastly, Lemma \ref{lem:qforms} implies that~$Q=0$.
\end{proof}

\subsection{Bipartite correlation matrices}
By definition of $\pi$ (recall \eqref{eq:projection}) and $\mathcal{E}_{n+m}$ we have:
 
\be\label{qdrgrt}
\cornm=\left\{C \in [-1,1]^{n\times m}: c_{ij}=\langle u_i, v_j\rangle, \  \text{ where  } \ \|u_i\|=\|v_j\|=1, \ \forall i,j\right\}.
\ee
For   a bipartite correlation matrix  $C\in \cornm$, any   $E\in \calE_{n+m}$ with  $\pi(X)=C$ is called a
  {\em  completion} of~$C$.
As  was shown in \cite{TS87},    all equality constraints in  \eqref{qdrgrt} can be relaxed with  inequalities, without enlarging the set.  For completeness we give a  proof in the~Appendix. 

\begin{lemma}[\cite{TS87}]\label{lem:unitnorm}
For all $n,m\ge 1$ we have that
\be\label{sdvfvergr} 
\cornm=\left\{C \in [-1,1]^{n\times m}: c_{ij}=\langle u_i, v_j\rangle, \  \text{ where  } \ \|u_i\|, \|v_i\|\le 1, \ \forall i,j\right\}.
\ee
\end{lemma}

Given a bipartite correlation  $C\in \cornm$, any family of  vectors $\{u_i\}_{i=1}^n,\{v_j\}_{j=1}^m$ satisfying $c_{ij}=\langle u_i, v_j\rangle$ and $  \ \|u_i\|, \|v_j\|\le 1, $ for all $i\in [n], j\in [m]$  is called a {\em C-system}.  
%
 
In the next result we summarize  some basic properties of  the set of completions of extreme bipartite  correlations. For completeness, we give  a short proof in the Appendix. 

\begin{lemma}[\cite{TS87}]\label{Tsirelson:equality} For any $C=(c_{ij})\in \ext(\cornm)$ we have that:
\begin{itemize}[itemsep=0.2em]
\item[$(i)$] All $C$-systems   necessarily consist of unit vectors; 
\item[$(ii)$] For any  $C$-system $\{u_i\}_{i=1}^n,\{v_j\}_{j=1}^m$ we have  that   ${\rm span}(\{u_i\}_{i=1}^n)={\rm span}(\{v_j\}_{j=1}^m)$;
\item[$(iii)$]  There exists a unique matrix  $E_C\in\calE_{n+m}$ satisfying $\pi(E_C)=C$. Furthermore,  we have  $E_C=\left(\begin{smallmatrix}A& C\\C^\sfT& B\end{smallmatrix}\right)\in\ext(\calE_{n+m})$ and     $\rank(E_C)=\rank(A)=\rank(B)=~\rank(C)$.
\end{itemize}
\end{lemma}

\subsection{Matrix factorizations of correlation matrices}\label{sec:matrixreps}
We are finally ready to show that any correlation matrix admits a matrix factorization as defined in Theorem~\ref{thm:bipartite}.

\begin{lemma}\label{lem:representations}
Consider a real symmetric matrix $E\in \calS_{n+m}$. The following are equivalent: 
\begin{itemize}[itemsep=0.2em]
\item[$(a)$] $E\in \calE_{n+m}$, i.e., there  exist  real unit vectors $\{a_i\}_{i=1}^n,\{b_j\}_{j=1}^m$  such that 
$$E={\rm Gram}(a_1,\ldots,a_n,b_1,\ldots,b_m).$$

\item[$(b)$] There exist $d\times d$  Hermitian matrices $\{A_i\}_{i=1}^n, \{B_j\}_{j=1}^m$ such that 
\medskip 
\bi[itemsep=0.2em]
\item[$(i)$] $E={\rm Gram}(A_1,\ldots,A_n, B_1,\ldots,B_m)$;
\item[$(ii)$]  $A_i^2=B_j^2={1\over d}I_d, \text{ for all }  i\in [n], j\in [m].$
\ei 
\medskip 
\item[$(c)$] There exist $d\times d $ Hermitian matrices 
\medskip 
$\{X_i\}_{i=1}^n, \{Y_j\}_{j=1}^m,  K$  such that 
\begin{itemize}[itemsep=0.2em]
 \item[$(i)$]  $E={\rm Gram}(KX_1,\ldots,KX_n, Y_1K,\ldots,Y_mK)$;
\item[$(ii)$]    $X_i^2=Y_j^2=I_d, \text{ for all }  i\in [n], j\in [m]$;
\item[$(iii)$]   $  \tr(K^2)=1$ and $  K$ is  positive definite.
\end{itemize}
\medskip 
\medskip 
\end{itemize} 
\end{lemma}

\begin{proof}$(a)\Longrightarrow (b).$
 Let $d=2^{\lfloor{ r/ 2}\rfloor}$. By  the properties  of the map $\gamma_r$ (recall~\eqref{eq:gamma})  we have that
$\la a_i,b_j\ra=\left\la { \gamma(a_i) \over \sqrt{d}},{\gamma(b_j) \over \sqrt{d}}\right\ra$ and  $\gamma(a_i)^2=\gamma(b_j)^2=I_d,$ for all  $i\in [n], j\in [m]$.

$(b)\Longrightarrow (c).$ Set  $K=d^{-1/2}I_d$  and   $X_i=\sqrt{d}A_i$, $Y_j=\sqrt{d}B_j$, for all  $i\in [n], j\in [m]$.


$(c)\Longrightarrow (a).$  For all $i\in [n]$ let $\tilde{a}_i={\rm vec}(K X_i)$ and set $a_i=({\rm Re}(\tilde{a}_i), {\rm Im}(\tilde{a}_i))$.  For all $j\in [m]$ define  $\tilde{b}_j$ and $b_j$  analogously.  As  the  entries  of $E$ are real numbers we have that   $E={\rm Gram}(a_1,\ldots,a_n,b_1,\ldots,b_m).$  Lastly,   note that 
$$1=\tr(K ^2)= \tr(X_i^2K ^2)=\langle  K X_i, K X_i\rangle=\|\tilde{a}_i\|^2=\|a_i\|^2, \text{ for all }  i\in [n].$$ Similarly we have that $\|b_j\|=1,$ for all $ j\in [m]$.
\end{proof}

We refer to any  family of matrices satisfying condition $(c)$ from Lemma~\ref{lem:representations} as a {\em matrix factorization} of $E$.  As already described in the introduction, our goal is to show that for extreme points of  $\calE_{n+m}$, we can place a lower bound on the size of matrix factorizations,  which is exponential in terms of $\rank(E)$. We note in passing that using the  same arguments we can also lower bound matrix factorizations  satisfying condition $(b)$ from Lemma~\ref{lem:representations}. Nevertheless,  lower bounds for matrix factorizations of type $(c)$ are stronger.

\subsection{Proof of the main technical result} In this section we prove Theorem~\ref{thm:bipartite}. 
This will follow as  a consequence  of  the following. 

\begin{lemma}\label{lemma:main}
Let   $E=\left(\begin{smallmatrix}A& C\\C^\sfT& B\end{smallmatrix}\right)$ be an extreme point of $\calE_{n+m}$ where $\rank(A)=\rank(E)=~n$. 
Consider a  family of $d\times d$ Hermitian operators $\{X_i\}_{i=1}^n$ satisfying 
\be\label{eq:basicidentity}
X_i^2=I_d,\  \text{ for all } i\in [n], \  \text{ and } \ 
 \Big(\sum_{i=1}^n\lambda_{ij}X_i\Big)^2=I_d, \ \text{ for all }  j\in [m],
\ee
where $\Lambda=(\lambda_{ij}) $ is  an $n\times m$ matrix satisfying  $C=A\Lambda$ and $B=\Lambda^\sfT A\Lambda$ (the fact that such a matrix exists follows from~\eqref{eq:usefulllll}). Then we have  that 
\be\label{Cliffordidentity}
\Big(\sum_{i=1}^n \mu_i X_{i}\Big)^{2} = (\mu^\sfT A\mu) I_d, \ 
 \text{ for all } \mu=(\mu_i)\in \R^n.
\ee
In particular, the algebra $\mathbb{C}[X_1,\ldots, X_n]$ generated  by $\{X_i\}_{i=1}^n$  is isomorphic to the rank-$n$ Clifford algebra $\mathcal{C}_n$ and thus, the size of the matrices $X_1,\ldots, X_n$ is lower bounded by~$2^{\lfloor n/2\rfloor}$.
\end{lemma}

\begin{proof}
Consider   vectors $\{a_i\}_{i=1}^n,\{b_j\}_{j=1}^m\subseteq \R^n$ satisfying    
$$E={\rm Gram}(a_1,\ldots,a_n,b_1,\ldots,b_m)\  \text{ and }\  \Span(\{a_i\}_{i=1}^n)=\R^n.$$
Using  \eqref{eq:useful} combined with  the fact that $C=A\Lambda$ and $B=\Lambda^\sfT A\Lambda$ we get    
\be\label{lem:aux}
b_j= \sum_{i=1}^n \lambda_{ij}a_i, \ \text{ for all } j\in [m].
\ee
For $i\in \{1,2\}$  we define  operator valued quadratic maps  $Q_i: \R^n\rightarrow \calh_d$ as 
$$Q_1\Big(\sum_{i=1}^n \mu_ia_i\Big):=\Big(\sum_{i=1}^n \mu_iX_i\Big)^2 \quad \text{ and } \quad Q_2\Big(\sum_{i=1}^n \mu_ia_i\Big):=\Big\|\sum_{i=1}^n \mu_ia_i\Big\|^2I_d.$$ 
As the vectors $\{a_i\}_{i=1}^n$ form a basis of $\R^n$ the maps $Q_1$ and $Q_2$ are well-defined. 
Note that the claim \eqref{Cliffordidentity} is equivalent to $Q_1=Q_2$. Thus, by  Theorem \ref{lem:equalityquadratic} it suffices to show~that 
\be\label{cdvgvrgr} 
Q_1(a_i)=Q_2(a_i),\ \forall i\in [n] \ \text{ and } \ Q_1(b_j)=Q_2(b_j),\  \forall j \in [m].
\ee
First, note that  
 $$Q_1(a_i)=X_i^2=I_{d}=Q_2(a_i), \text{ for all } i \in [n],$$
 where we use  that $\|a_i\|=1,$ for all $i \in [n]$.
 Furthermore,  for all $j\in [m]$ we have 
 $$Q_1(b_j)=Q_1\Big(\sum_{i=1}^n \lambda_{ij}a_i\Big)=\Big(\sum_{i=1}^n \lambda_{ij}X_i\Big)^2=I_d,
$$
where for the  first equality we use \eqref{lem:aux} and  for the last equality \eqref{eq:basicidentity}. Similarly,  
 $$Q_2(b_j)=Q_2\Big(\sum_{i=1}^n \lambda_{ij}a_i\Big)=\Big\|\sum_{i=1}^n \lambda_{ij}a_i\Big\|^2I_d=\|b_j\|^2I_d=I_d,$$
 where we use 
  that $\|b_j\|=1,$ for all $j\in [m]$. Thus, \eqref{cdvgvrgr} holds which  in turn implies  \eqref{Cliffordidentity}. 
  
  
  Lastly, as an immediate consequence of  \eqref{Cliffordidentity} we get  that
  \be 
  X_iX_{j}+X_{j}X_i=2 A_{ij}I_d, \ \text{ for all } i,j\in [n].
  \ee
Let   $A=\sum_{k=1}^n\lambda_k u_ku_k^\sfT$ be a spectral decomposition of $A$. By assumption $A$ is positive definite and thus, $\lambda_k>0 $ for all $k\in [n]$. Setting 
$$X_k':=\lambda_k^{-1/2}\sum_{i=1}^nu_k(i)X_i,\ \text{ for all }  k\in [n],$$
 we have that   
$$X_i'X_j'+X_j'X_i'=2\delta_{i,j}I_d, \ \text{ for all }  i,j\in [n],$$
and thus    $\mathbb{C}[ X_1,\ldots, X_n]$ is isomorphic to the rank-$n$ Clifford algebra $\mathcal{C}_{n}$.%
  \end{proof}

We  now   give the proof of Theorem~\ref{thm:bipartite}. We restate it below for the ease of the reader.

\begin{theorem}
Let   $E=\left(\begin{smallmatrix}A& C\\C^\sfT& B\end{smallmatrix}\right)$ be an extreme point of $\calE_{n+m}$ where  $\rank(A)=\rank(E)=~n$. 
Consider  $d\times d$   Hermitian matrices    $\{X_i\}_{i=1}^n, \{Y_j\}_{j=1}^m,  K$~satisfying
\begin{itemize}[itemsep=0.2em]
 \item[$(i)$]  $E={\rm Gram}(KX_1,\ldots,KX_{n}, Y_1K,\ldots,Y_mK)$;
\item[$(ii)$]    $X_i^2=Y_j^2=I_d,\  \text{ for all }  i\in [n], j\in [m]$;
\item[$(iii)$]   $  \tr(K^2)=1$ and $  K$ is  positive definite.
\end{itemize}Then, the algebra $\mathbb{C}[ X_1,\ldots,X_n]$   is isomorphic to the rank-$n$ Clifford algebra $\mathcal{C}_{n}.$ In particular,  the size of the matrices  $X_1,\ldots, X_n$ is      lower bounded by~$2^{\lfloor n/2\rfloor}$. 
\end{theorem}

\begin{proof}
As $\rank(E)=\rank(A)$, there exists   an $n\times m$ matrix  $\Lambda=(\lambda_{ij}) $  such that $C=A\Lambda$ and $B=\Lambda^\sfT A\Lambda$.  
Since   $ E={\rm Gram}(KX_1,\ldots,KX_{n}, Y_1K,\ldots,Y_mK)$ it follows by \eqref{eq:useful} that    
$$Y_{j}K=\sum_{i=1}^n\lambda_{ij} KX_i, \ \text{ for all }  j\in [m],$$ 
and as $K$ is positive definite (and hence invertible) we obtain 
\be\label{cdfefre}
Y_{j}=\sum_{i=1}^n\lambda_{ij} KX_iK^{-1},\  \text{ for all } j\in [m].
\ee
Lastly, define  
$$\tilde{X}_i=KX_iK^{-1}, \text{ for all } i\in [n].$$  By  assumption  we have that  $X_i^2=I_d$, for all $i\in [n]$, which implies that   $\tilde{X}_i^2=I_d$. 
Furthermore, as  $Y_j^2=I_d$  for all $j\in [m],$  it follows by  \eqref{cdfefre}   that 
\be\label{xcdsvdegvreg}
 \Big(\sum_{i=1}^n \lambda_{ij}\tilde{X}_i\Big)^2=I_d, \  \text{ for all }   j\in [m].
 \ee
 The proof  of the theorem is is concluded using  Lemma~\ref{lemma:main}.
 \end{proof}
%
%
%

\section{Cpsd  matrices  with sub-exponential cpsd-rank}\label{sec:cpsdrank}
In this section we use Theorem~\ref{thm:bipartite} to prove Theorem~\ref{thm:cpsdlowerbound}. The crux of the proof lies in the following  result. 

%


\begin{theorem}\label{cpsd:lowerbound}
For any     $C=(c_{ij})\in \ext (\calE_n)$ the matrix 
\be\label{cdfdrf}
P_C=\sum_{i,j=1}^n{1\over 4} \begin{pmatrix} 1+c_{ij} & 1-c_{ij} \\ 1-c_{ij} & 1+c_{ij}\end{pmatrix}\otimes e_ie_j^\sfT,
\ee
is cpsd and furthermore,   $\CPSDR(P_C)\ge 2^{\lfloor {\rank(C)/ 2} \rfloor}$.
 
\end{theorem}

\begin{proof}Let $C=\gram(\{u_i\}_{i=1}^n)$ where $\{u_i\}_{i=1}^n\subseteq \R^r$ and $\|u_i\|=1, \forall i\in [n]$.  As suggested by \eqref{cdfdrf} we  think of $P_C$ as an $n\times n$ block matrix where each block has size  $ 2\times 2 $ and is  indexed by $\{\pm 1\}$. 

We first show that $P_C\in \CPSD^{2n}$.  For this,  set $d=2^{\lfloor{ r/ 2}\rfloor}$  and define 
\be
\Gamma^i_a={I+a\gamma_r(u_i)\over 2\sqrt{d}}, \quad  \text{ for all } i\in [n], a\in \{\pm 1\},
\ee
and note that by the properties of the  $\gamma_r$ map (recall \eqref{eq:gamma}), these matrices  are Hermitian psd. Furthermore, by direct calculation for all $i,j \in [n]$ and $a,b \in  \{\pm 1\}$ we  have that
\be \label{cdvdr}
\la \Gamma^i_a,\Gamma^j_b\ra = {d+ab\la \gamma_r(u_i),\gamma_r(u_j)\ra\over 4d}= {1+ab\la u_i,u_j\ra\over 4}= {1+abc_{ij} \over 4},
\ee
which shows that the matrices $\left\{\Gamma^i_a: i\in [n], a\in \{\pm 1\}\right\}$ form a  cpsd-factorization for $P_{C}$. Next we proceed to show the lower bound. 

 Let
 $\left\{P^i_a:  i\in [n], a\in \{\pm 1\}\right\}$ be a size-optimal  cpsd-factorization for $P_C$. We now identify  some useful properties of these matrices which  we use  later  in the proof.  As the entries of $P_C$  in each $2\times 2$ block sum up to one we~get  
\be\label{cdvrgr}
\sum_{a\in \{\pm 1\}} P^i_a=\sum_{a\in \{\pm 1\}} P^{j}_a, \quad \text{ for all } i,j\in [n].
\ee
For all $i\in [n]$ set    
\be\label{def:k}
K:=\sum_{a\in \{\pm 1\}} P^i_a,
\ee  which  is well-defined by \eqref{cdvrgr}. Furthermore,  note that $K$ is psd and $ \la K,K\ra=~1.$
 
 Since the cpsd factorization is  size-optimal  we may  assume without loss of generality that $K$ is  diagonal and positive definite. Indeed, let $K=Q\Lambda Q^*$ be its spectral decomposition. Clearly, the matrices 
$\left\{ Q^*P^i_aQ: i\in [n], a\in \{\pm 1\}\right\}$ are Hermitian positive semidefinite and 
 as $Q$ is unitary, it follows that they form a cpsd-factorization for  $P_C$. As a consequence, if $K$ was rank-deficient,  by restricting the matrices $\left\{ Q^*P^i_aQ: i\in [n], a\in \{\pm 1\}\right\}$ onto  the support of $K$, we would get another cpsd-factorization of smaller size. This contradicts the  assumption  that  $\left\{P^i_a:  i\in [n], a\in \{\pm 1\}\right\}$ was size-optimal. 
 
Our next goal is to use the cpsd-factorization  $\left\{P^i_a:  i\in [n], a\in \{\pm 1\}\right\}$ to obtain the matrix  factorization  to which Theorem~\ref{thm:bipartite} will be applied.  As  $K$ invertible  we have  that 
\be\label{cdfvfbgf} 
\la P^i_a,P^j_b\ra=\la K (K^{-1/2}P^i_aK^{-1/2}),(K^{-1/2}P^j_bK^{-1/2})K\ra=\la K\tilde{P}^i_a,\tilde{P}^j_bK\ra, \quad \forall i,j\in [n],
\ee
where we define 
\be\label{eq:rgrvrtgt}
\tilde{P}^i_a:=K^{-1/2}P^i_aK^{-1/2}, \quad \text{ for all } i\in [n], a\in \{\pm 1\}.
\ee
An easy calculation shows that
\be\label{cwcverfer}
c_{ij}=\sum_{a,b\in \{\pm 1\}}ab\left({1+abc_{ij}\over 4}\right)=\sum_{a,b\in \{\pm 1\}}ab\la P^i_a,P^j_b\ra=\sum_{a,b\in \{\pm 1\}}ab\la K\tilde{P}^i_a,\tilde{P}^j_bK\ra,
\ee
where for the second equality we use that $\left\{P^i_a:  i\in [n], a\in \{\pm 1\}\right\}$ is a   cpsd-factorization for $P_C$ and  the third  equality follows from \eqref{cdfvfbgf}. 
 Setting  $$X_i:=\tilde{P}^i_1-\tilde{P}^i_{-1},\quad \forall i\in [n],$$
it follows by \eqref{cwcverfer} that
\be\label{dfvrbhrt} 
c_{ij}=\la KX_i, X_jK\ra, \quad \forall i,j\in [n].
\ee

By \eqref{def:k} we have  $\sum_{a \in \{\pm 1\}} P^i_a=K$ which implies that   $\tilde{P}^i_1+\tilde{P}^i_{-1}=I, $ for all $i\in [n].$  Thus, for any $i\in [n]$, the Hermitian matrix  $X_i=\tilde{P}^i_1-\tilde{P}^i_{-1}=2\tilde{P}^i_1-I$ has spectrum in $[-1,1]$, {\em i.e.}, $X_i^2\preceq I$.  In fact, as $C\in \ext (\calE_n)$, it follows by Lemma~\ref{Tsirelson:equality} $(i)$  that  
\be\label{cdfgergre}
{X}_i^2=I, \quad \text{ for all }  i\in [n].
\ee
Note that the same argument was given in the proof of Theorem~\ref{thm:lowerboundlocaldimension}.



We are now ready to conclude the proof. As     $C\in \ext(\calE_n)$,  Theorem \ref{Tsirelson:equality} $(iii)$  implies that  $C$ has a unique elliptope completion   $E_C\in\calE_{2n}$, which  moreover is an extreme point of $\calE_{2n}$. Nevertheless,    as $C\in \calE_n$,   the matrix $ \left(\begin{smallmatrix}C & C\\C & C \end{smallmatrix}\right)$ is clearly an elliptope completion of $C$. As a consequence we have that 
\be 
E_C=\left(\begin{matrix}C & C\\C & C \end{matrix}\right)\in \ext(\calE_{2n}),
\ee
which is the matrix to which we will apply Theorem~\ref{thm:bipartite}. The last step is to exhibit a matrix factorization for $E_C$. For this consider the psd matrix 
$$E_C':={\rm Gram}(KX_1,\ldots,KX_n,X_1K,\ldots,X_nK).$$ 
By \eqref{cdfgergre} we~have   
$$ \la KX_i,KX_i\ra=\tr(X_i^2K^2)=\tr(K^2)=1,\quad  \forall i\in [n],$$ 
and thus, $E_C'$ is an element of the elliptope $\calE_{2n}$. Finally,  by \eqref{dfvrbhrt} it follows that $E_C'$ is an elliptope  completion of $C$. Thus, again  by  Lemma \ref{Tsirelson:equality} $(iii)$ we get that   $E_C=E_C'$, {\em i.e.},
$$\left(\begin{matrix}C & C\\C & C \end{matrix}\right)={\rm Gram}(KX_1,\ldots,KX_n,X_1K,\ldots,X_nK)\in \ext(\calE_{2n}),$$
and the  proof is concluded by  Theorem~\ref{thm:bipartite}.
\end{proof}

 To prove  Theorem~\ref{thm:cpsdlowerbound}, it remains to combine Theorem~\ref{cpsd:lowerbound}  with the following well-known fact: 
For any $X\in \ext (\calE_n)$   we have that $\rank(X)\le r_{\max}(n)$. Furthermore, for any integer $r\in [1,r_{\max}(n)]$ there exists a matrix $X_r\in \ext (\calE_n)$ with $r=\rank(X_r)$ \cite{GPW}.

\section{Relation to Tsirelson's work}\label{sec:quantumcorrelations}

In this section we  explain the connection between quantum information theory  and bipartite correlation matrices. The set $\cornm$ was  studied by Tsirelson due to its relevance  to quantum information theory. Algebraically, this is captured by the  following result found in \cite[Theorem 2.1]{TS87}. We give a brief   proof for~completeness.

\begin{theorem}\label{thm:tsirelson1appendix} Let $C=(c_{ij})\in  [-1,1]^{n\times m}$. Then,    $C\in \cornm$ if and only if 
there exist  Hermitian matrices    $\{M_i\}_{i=1}^n, \{N_j\}_{j=1}^m\subseteq \calh_d$  and a Hermitian matrix  $\rho \in \calh_{d^2}$ such that 
\bi[itemsep=0.2em] 
\item[$(i)$] $M_i^2\preceq I,\  N_j^2\preceq I$, for all $i\in [n], j\in [m];$
\item[$(ii)$]  $\rho$ is positive semidefinite  with $\tr(\rho)=1 $;
\item[$(iii)$]  $c_{ij}=\tr\left((M_i\otimes N_j)\rho\right),$ for all $ i\in [n], j\in [m].$
\ei
\end{theorem}

\begin{proof} Let  $C\in \cornm$ and consider vectors  $\{u_i\}_{i=1}^n,\{v_j\}_{j=1}^m\subseteq \R^r$ satisfying $c_{ij}=\langle u_i, v_j\rangle$ and  $   \|u_i\|, \|v_j\|\le 1,$ for all  $i\in [n], j\in [m]$ (such vectors exist by \eqref{sdvfvergr}). Then,   
\be\label{cdfvgrthtyhty}
c_{ij}=\la u_i,v_j\ra={\tr\left( \gamma_r(u_i)\gamma_r(v_j)\right)\over d} =\tr \left((\gamma_r(u_i)\otimes \gamma_r(v_i)^\top)\psi_d\psi_d^*\right),\ \forall  i\in [n], j\in [m],
\ee
where  $d=2^{\lfloor{ r/ 2}\rfloor}$ and $\psi_d:=d^{-1/2}\sum_{i=1}^de_i\otimes e_i\in \C^d\otimes \C^d$. To prove \eqref{cdfvgrthtyhty}, for the second  equality  we use \eqref{eq:gamma} and for the third  one
that   $\psi_d^*(A\otimes B)\psi_d= \frac{1}{d} \, \tr\left(AB^\sfT\right),\ \forall A,B\in~\mathcal{M}_d.$

Conversely, consider matrices $\{M_i\}_{i=1}^n, \{N_j\}_{j=1}^m$ and $\rho$ satisfying $(i), (ii)$ and $(iii)$. 
 Setting  $A_i=M_i\otimes I$ and  $B_j=I\otimes N_j$  for all $i\in [n], j\in [m]$ we get that 
$$c_{ij}=\tr(A_iB_j\rho)=\langle \rho^{1/2}A_i, \rho^{1/2}B_j\rangle, \ \forall i\in [n], j\in [m].$$ 

For all $i\in [n]$ let $\tilde{u}_i={\rm vec} (\rho^{1/2}A_i)$ and  $u_i=({\rm Re}(\tilde{u}_i), {\rm Im}(\tilde{u}_i))$.  For all $j\in [m]$ define $\tilde{v}_j$ and $v_j$  analogously.  As the entries of $C$ are real numbers we get that     $c_{ij}=\langle u_i, v_j\rangle, $ for all  $i\in [n], j\in [m] $. Lastly,   note that 
$$ \|u_i\|^2=\|\tilde{u}_i\|^2=\tr(A_i^2\rho)\le \tr(\rho)=1,$$
where for the inequality we  use that $ A_i^2\preceq I.$ Similarly, we get  $\|v_j\|\le 1, \ \forall j \in~[m]$.
\end{proof}

The algebraic representation of the set of bipartite correlations given above  turns out to have {\em operational interpretation} 
within the context of   quantum information theory.

\begin{defn}\label{rgrgtr}
A matrix  $C=(c_{ij})\in  [-1,1]^{n\times m}$ is called a {\em quantum   correlation matrix} if  there exist 
Hermitian matrices    $\{M_i\}_{i=1}^n, \{N_j\}_{j=1}^m\subseteq \calh_d$  and a Hermitian matrix  $\rho \in \calh_{d^2}$ 
(for some $d\ge 1$) satisfying conditions $(i), (ii)$ and $(iii)$ from Theorem~\ref{thm:tsirelson1appendix}. 
\end{defn}

We refer  to any such  family of matrices   as a  {\em tensor product  representation} of $C$ with {\em local dimension $d$}. In this section we  use Theorem~\ref{thm:bipartite} to lower bound   the local dimension of 
tensor product  representations corresponding to extreme points of the set of quantum~correlations. 

As a first step we show that without loss of generality,  we may only consider tensor product representations  where  $\rho$ is a rank-one matrix.  This is known but we   give a  short proof for completeness.

\begin{lemma}[\cite{SV}] \label{lem:scwefew} For any    $C\in \ext(\cornm)$, the minimum local dimension  of   a  tensor product representation can be  achieved  by a rank-one representation $\rho=\psi\psi^*$ satisfying: 
\be\label{eq:specialform}
\psi=\sum_{i=1}^{d}\lambda_i e_i\otimes e_i^* \in \C^{d}\otimes \C^{d}, \quad  \lambda_i>0 \ (\forall i\in [d]), \quad  \text{ and } \quad \sum_{i=1}^{d} \lambda_i^2=1.
\ee
 \end{lemma} 
 
 \begin{proof}The extreme points of the compact convex set  $\{\rho: \rho\succeq 0, \ \tr(\rho)=1\}$  are matrices of  the form $\phi\phi^*$, where $\|\phi\|=1$. 
 Thus, by the extremality assumption,  for every tensor product representation  of $ C$ we     have that $\rho= \phi\phi^*$, for some vector $\phi$ with   $\|\phi\|=1$. 
  
 It remains to show that given  a rank-one  tensor product representation of $C$ with local  dimension $d'$, {\em i.e.},
 $$c_{ij}=\tpsi^*(M_i\otimes N_j)\tpsi, \ \forall i,j \  \text{ where }\ \phi\in \C^{d'}\otimes \C^{d'},$$
  we can construct another rank-one  tensor product representation of $C$  satisfying \eqref{eq:specialform}, whose  local dimension is upper bounded by $d'$.  
  
  For this, 
  consider  a Schmidt decomposition   of $\tpsi$, {\em i.e.},  
  $\tpsi=\sum_{k=1}^{d} \lambda_k x_k\otimes y_k,$ where   $\{ \lambda_k\}_{k=1}^{d}$ are strictly positive, $\sum_{k=1}^{d}\lambda_k^2=1$, and   $\{ y_k\}_{k=1}^{d}, \{ x_k\}_{k=1}^{d}\subseteq\mathbb{C}^{d'}$ are orthonormal vectors. Clearly we have  that $d\le d'$.  Define the $d\times d'$ matrices     $U := \sum_{k=1}^{d} \, e_k x_k^*$ and  $V:= \sum_{k=1}^{d}e_ky_k^*$,  and note that  the vector 
  $$\psi:= (U \otimes V) {\tpsi}=\sum_{k=1}^{d}
\lambda_k {e_k}\otimes {e_k} \in \C^{d} \otimes
\C^{d},$$ satisfies $\tr(\psi\psi^*)=1$. Moreover, as $UU^*=VV^*=I_{d},$ it follows that the Hermitian $d\times d$ matrices  $ \{\tilde{M}_i:=U M_{i} U^{*}\}_{i=1}^n$  and $ \{\tilde{N}_j:=V N_{j}V^*\}_{j=1}^m$ have spectrum in $[-1,1]$. 
To see this,  recall that for  any Hermitian matrix $X$, the condition $X^2\preceq I$ is equivalent to $\left(\begin{smallmatrix}I & X\\X& I\end{smallmatrix}\right)\succeq 0.$ ({\em e.g.} by using Schur complements). By assumption we have that $M_i^2\preceq I$ and $N_j^2\preceq I$, for all $i,j$.  Thus,   we have that $\left(\begin{matrix}I_{d'} & M_i\\M_i& I_{d'}\end{matrix}\right)\succeq 0$ which implies 
$$\left(\begin{matrix}I_{d} & \tilde{M}_i\\\tilde{M}_i& I_{d}\end{matrix}\right)= \left(\begin{matrix}U & 0\\0 & U\end{matrix}\right)\left(\begin{matrix}I_{d'} & M_i\\M_i& I_{d'}\end{matrix}\right) \left(\begin{matrix}U^* & 0\\0 & U^*\end{matrix}\right)\succeq 0.$$
Similarly, we get $\tilde{N}_j^2\preceq I,\ \forall j\in [m].$ Lastly, an easy calculation gives   that 
$$c_{ij}=\tpsi^*({M}_i\otimes {N}_j)\tpsi=\psi^*(\tilde{M}_i\otimes \tilde{N}_j)\psi, \quad  \forall i\in [n], j\in [m],$$ and the proof is concluded.
\end{proof}

As an application application of Theorem~\ref{thm:bipartite}  we now prove Theorem \ref{thm:lowerboundlocaldimension}.


\begin{proof}({\bf of Theorem \ref{thm:lowerboundlocaldimension}}) By Lemma~\ref{lem:scwefew} we may only consider rank-one tensor product representations, {\em i.e.,} 
$c_{ij}=\psi^*(M_i\otimes N_j)\psi,$  where $M_i^2\preceq I, N_j^2\preceq I$, for all $i\in [n], j\in [m]$
 and $\psi$ has the form given  in  \eqref{eq:specialform}. 
Set $K:={\rm vec}({\psi})=\sum_{i=1}^d\lambda_i e_ie_i^*$ and note that $K$   is  positive definite (and even  diagonal)  and  satisfies $\tr(K^2)=1$. By \eqref{vecprop}  we have that 
\be\label{xcsdvdegrtg}
c_{ij}={\rm vec}(K)^*({M}_i\otimes{N}_j){\rm vec}(K)=\tr(K{M}_iK{N}_j^\sfT)=\la K{X}_i,Y_jK\ra,
\ee
where $X_i:={M}_i$ and $Y_j:={N}^\sfT_j$.  Note that since  $N_j$ is Hermitian the same holds for $Y_j$.

Clearly $X_i^2\preceq I$,  and since  a matrix and its transpose have the same eigenvalues we also have that $Y_j^2\preceq I$. We now show that in fact $X_i^2=Y_j^2=I$, for all $i\in [n],j\in [m]$. Towards a contradiction, assume there exists $i^*\in [n]$ such that $X_{i^*}^2\prec I$. Then, 
$$\la KX_{i^*},KX_{i^*}\ra=\tr(X_{i^*}^2K^2)<\tr(K^2)=1.$$
Thus, by vectorizing the matrices $\{K{X}_i\}_{i=1}^n$ and $\{Y_jK\}_{j=1}^m$, in view of \eqref{xcsdvdegrtg} we get a $C$-system  where one of the vectors  has norm strictly less than one. Nevertheless, as $C\in \ext(\cornm)$, this possibility has been already  excluded in Lemma~\ref{Tsirelson:equality} $(i)$. 

Since  $X_i^2=Y_j^2=I$, for all $i\in [n],j\in [m]$, the matrix $\gram(\{KX_i\}_i,\{Y_jK\}_j)$ is an elliptope  completion of $C$. On the other hand,  since $C\in \ext(\cornm)$,  we have seen in   Lemma~\ref{Tsirelson:equality} $(iii)$ that $C$ has a unique  completion $E_C= \left(\begin{smallmatrix}A& C\\C^\sfT& B\end{smallmatrix}\right)$ where    $\rank(E_C)=\rank(A)=\rank(B)=\rank(C)$ and $E_C\in \ext (\calE_{n+m})$. Consequently, we have that 
$$E_C=\gram(\{KX_i\}_i,\{Y_jK\}_j),$$ and the claim follows by applying  Theorem \ref{thm:bipartite}. 
\end{proof}

We note that Theorem~\ref{thm:lowerboundlocaldimension} essentially follows from Tsirelson's  seminal  work \cite{TS87}, although it is not explicitly stated there (it is mentioned in \cite{TS93} albeit without proof). Indeed,  in \cite{TS87} Tsirelson studies the properties of another family  of matrix representations of quantum correlations called {\em commuting representations},  in the case where  the ambient Hilbert space is finite-dimensional or countably infinite.  His main result is that  for any $C \in \ext(\cornm)$, the   matrices in  a (nondegenerate) commuting representation     correspond to  a  representation  of  an appropriate Clifford algebra \cite[Theorem 3.1]{TS87}. As  a consequence,   the dimension of any commuting representation of $C \in \ext(\cornm)$ is lower bounded by $4^{\lfloor \rank(C)/2\rfloor}$ (for a concise  proof of this fact see \cite[Theorem 4.4]{cpsd2} or \cite[Theorem 25]{CPSD}). On the other hand, it is well-known and easy to see  that any tensor product representation  with local dimension $d$ gives rise to a commuting representation of size $d^2$. Putting everything together we arrive at Theorem~\ref{thm:lowerboundlocaldimension}.

 Interestingly,   Theorem \ref{thm:bipartite}  generalizes  Theorem~\ref{thm:lowerboundlocaldimension} since  there exist  
matrices $\left(\begin{smallmatrix}A& C\\C^\sfT& B\end{smallmatrix}\right)\in\ext(\calE_{n+m})$ for which  $C\not \in \ext (\cornm)$. 
To give   a concrete  example define
$$u_{ii}=e_i \quad (1\le i\le r), \quad \text{ and } \quad u_{ij}=\frac{e_i+e_j}{\sqrt{2}}\quad (1\le i<j\le r),$$
and let $E$ be the Gram matrix of the $u_{ij} \ (1\le i\le j\le r)$ ordered lexicographically.  
 Using \eqref{eq:extremality} one can easily verify   that  $E$ is an extreme point of the ${\binom{r+1}{2}}$-dimensional elliptope. On the other  hand,  let $C$ be the submatrix  of $E$ obtained by restricting to the first $r$ rows and the columns indexed by pairs $(i,j)$ in the range $2\le i\le j\le r$. Clearly,   $C$ is a $r\times \binom{r}{2}$ bipartite correlation matrix  but   it is an  not extreme point of  $\pi\left(\calE_{r+\binom{r}{2}}\right)$. This is an immediate consequence  of  Lemma~\ref{Tsirelson:equality}  $(ii)$,  since
$e_1\not \in \Span\left(u_{ij}:2\le i\le j\le r\right).$

\bibliographystyle{abbrv}
\bibliography{biblio,biblio1}

\appendix
\section{Omitted  proofs}

\subsection{Proof of Lemma \ref{lem:unitnorm}.}

\begin{proof} Let $S$ denote the set in the right hand side of \eqref{sdvfvergr}. Clearly,  we have  $\cornm\subseteq~S$. As $S$ is a compact convex set, for the converse  inclusion it suffices to show  that every extreme point of $S$ necessarily satisfies all the norm inequalities with equality. For this let $C\in \ext(S)$ and assume towards a contradiction   that $\|u_1\|<1$.  Select $\delta \in \R$ such that $\|u_1(1 \pm \delta) \|\le 1$. Consider  the $n\times m$ matrices $C^+,C^-$ where $C^{\pm}_{ij}=\la u_i,v_j\rangle,$ for all $i$ and $j\ne 1$,  and  $C^{\pm}_{1j}=\la u_1(1\pm \delta),v_j\rangle, $ for all $j\in [m]$. Clearly, $C^+,C^-\in S$ and  by definition  $C=(C^+ + C^-)/2$. Thus, since $C^+,C^-\ne C$ we contradict the fact that $C\in \ext(S)$.
\end{proof}

\subsection{Proof of Lemma \ref{Tsirelson:equality}.}

\begin{proof} $(i)$ This was shown already in   the proof of Lemma~\ref{lem:unitnorm}. 

\medskip

 $(ii)$ 
 Consider a $C$-system  $\{u_i\}_{i=1}^n,\{v_j\}_{j=1}^m$.  We only show   ${\rm span}(\{u_i\}_{i=1}^n) \subseteq {\rm span}(\{v_j\}_{j=1}^m),$ the other inclusion follows similarly.  Towards a contradiction, say there exists some ${i^*\in[n]}$ such that $u_{i^*}\not \in {\rm span}(\{v_j\}_{j=1}^m)$. Let  $P$ be the orthogonal projector onto ${\rm span}(\{v_j\}_{j=1}^m)$.  As $\|Pu_{i^*}\|<\|u_{i^*}\|\le 1,$  
 the vectors  $(\{u_i\}_{i=1}^n\setminus \{u_{i^*}\})\cup \{P u_{i^*}\}$ and $  \{v_j\}_{j=1}^m$  form  a  new $C$-system.  Lastly, since  $C\in \ext(\cornm)$ and $\|P u_{i^*}\|< 1$, this contradicts case $(i)$. 
 
 \medskip 
 
 $(iii)$  Consider a $C$-system  $\{u_i\}_{i=1}^n,\{v_j\}_{j=1}^m$.
 As $C\in \ext(\cornm)$, 
 by case $(i)$ 
 we have that   $\|u_i\|=\|v_j\|=1,$ for all $ i\in [n],  j\in [m].$ This shows that $C\in \cornm$ has at least one completion in $\calE_{n+m}$. 
 The next step is to show that   $C$  it has a unique completion in $\calE_{n+m}$. For this, let ${\rm Gram}(\{u'_i\}_i,\{v'_j\}_j)$ and ${\rm Gram}(\{u''_i\}_i,\{v''_j\}_j)$ be two elliptope completions. For $i\in [n], j\in [m]$ define $u_i={u'_i\oplus u''_i\over \sqrt{2}}$ and $v_j={v'_j\oplus v''_j\over \sqrt{2}}$ and note that  they form a $C$-system.  Thus, by case $(ii)$  we have that ${\rm span}(\{u_i\}_{i=1}^n)={\rm span}(\{v_j\}_{j=1}^m). $ In particular, for all $i\in[n]$ there exist scalars $\{\lambda^i_j\}_{j=1}^m$ satisfying $u_i=\sum_{j=1}^m\lambda^i_jv_j$. By the definition of $u_i$ and $v_j$ this implies that 
$u'_i=\sum_{j=1}^m\lambda^i_jv'_j $  and $ u''_i=\sum_{j=1}^m\lambda^i_jv''_j.$ Then, for all $i,i'\in [n]$ we get  
 $$\la u'_i,u'_{i'}\ra=\sum_{j=1}^m\lambda^i_j\la v'_j ,u'_{i'}\ra=\sum_{j=1}^m\lambda^i_j\la v''_j ,u''_{i'}\ra=\la u''_i,u''_{i'}\ra.$$
 Analogously it follows that for all $j,j'\in [m]$ we have $  \la v'_j,v'_{j'}\ra=\la v''_j,v''_{j'}\ra$. Putting everything together we get that ${\rm Gram}(\{u'_i\}_i,\{v'_j\}_j)$ and ${\rm Gram}(\{u''_i\}_i,\{v''_j\}_j)$. 
   
  For any  $C\in \ext(\cornm)$ we denote by $E_C$ its unique elliptope completion.  
  As the set of all completions of $C$ is a face of $\calE_{n+m}$ it follows that $E_C\in \ext(\calE_{n+m})$ (here we use the fact  that the only way for a single point to be a face  is for the point itself to be   extreme). 
 
 Lastly, we have already seen that $\rank(E_C)=\rank(A)=\rank(B)$.  Clearly $\rank(E_C)\ge \rank(C)$ and it remains to   show that $r:=\rank(A)\le \rank(C).$ Wlog  assume that the first $r$ rows of $A$ are linearly independent. Then,  $\sum_{i=1}^r\lambda_i \la u_i,v_j\ra=0,$ for all  $j\in [m]$ implies that  $\sum_{i=1}^r\lambda_i  u_i=0$ and thus $\lambda_i=0$, for all $i\in [r].$
\end{proof}

\paragraph{{\bf Acknowledgments.}}
Both authors are supported in part by the
Singapore National Research Foundation under NRF RF Award
No.~NRF-NRFF2013-13.


%
%
%
%
%
%


\end{document}